\documentclass[review]{elsarticle}
\usepackage[utf8]{inputenc}

\usepackage{amsmath} 
\usepackage{amssymb}

\newcommand{\firststirlingnumber}[2]{\ensuremath{\genfrac{[}{]}{0pt}{}{#1}{#2}}}

\newcommand{\Sym}                 [1] {\ensuremath{\mathfrak{S}_{#1}}}
\newcommand{\Symdist}                 [2] {\ensuremath{\mathfrak{S}_{#1, #2}}}

\DeclareMathOperator*{\argmax}{argmax}

\input{newtheorems-en}

\usepackage [     backgroundcolor=white,linecolor=black]    {todonotes}

\usepackage{cleveref} 
\Crefname{conjecture}{Conjecture}{Conjectures}
\Crefname{observation}{Observation}{Observations}

\title{Asymptotic normality and combinatorial aspects of the prefix exchange distance distribution}
\author{Simona Grusea}
\ead{grusea@insa-toulouse.fr}
\address{Institut de Math\'ematiques de Toulouse, INSA de Toulouse, Universit\'e de Toulouse, France}
\author{Anthony Labarre}
\ead{Anthony.Labarre@u-pem.fr}
\address{Universit\'e Paris-Est, LIGM UMR CNRS 8049 \\
Universit\'e Paris-Est Marne-la-Vall\'ee\\
5 boulevard Descartes \\
77420 Champs-sur-Marne, France.\\
}

\begin{document}

\begin{abstract}
The prefix exchange distance of a permutation is the minimum number of exchanges involving the leftmost element that sorts the permutation. We give new combinatorial proofs of known results on the distribution of the prefix exchange distance for a random uniform permutation. We also obtain expressions for the mean and the variance of this distribution, and finally, we show that the normalised prefix exchange distribution converges in distribution to the standard normal distribution.
\end{abstract}


\maketitle

\section{Introduction}

An ever-growing body of research has been devoted to the study of various measures of disorder on permutations, with the intention of expressing how many elementary operations (whose type may vary but which are fixed beforehand) they should undergo in order to become sorted. One of the earliest examples of such a measure is the \emph{Cayley distance}, which corresponds to the minimum number of transpositions that must be applied to a permutation in order to obtain the identity permutation. This distance is easily expressed in terms of the number of \emph{cycles} of the permutation~\cite{Cayley1849}, and the \emph{signless Stirling numbers of the first kind} can be used to characterise exactly the distribution of the Cayley distance --- i.e., the number of permutations of $n$ elements with Cayley distance $k$. Motivations for studying these distances and their distributions outside pure mathematical fields include the study of sorting algorithms~\cite{Estivill-Castro1992}, genome comparison~\cite{fertin-combinatorics}, and the design of interconnection networks~\cite{Lakshmivarahan1993}.

We focus in this paper on the \emph{prefix exchange} operation,  a restricted kind of transposition that swaps any element of a permutation with its first element. This operation was introduced by \citet{Akers1989}, who also gave a formula for computing the associated \emph{prefix exchange distance}, i.e., the minimum number of prefix exchanges required to transform a given permutation into the identity permutation. \citet{portier-whitney} later succeeded in obtaining the generating function of the corresponding distribution, which they then used to derive an explicit formula (with subsequent corrections by \citet{shen-whitney}) as well as recurrence formulas for computing the so-called ``Whitney numbers of the second kind for the star poset'', i.e., the number of permutations of size $n$ with prefix exchange distance $k$ (see \citet{oeis-pexc} for a table with the first few terms).

We revisit in this paper the results obtained by \citet{portier-whitney} by taking the opposite direction: we first obtain new proofs for their exact and recurrence formulas, and then use these formulas to recover their expression for the generating function. Our proofs are purely combinatorial, a desirable property since such proofs are often simpler in addition to providing new insight into the underlying objects~\cite{Bona2012,Stanley2012}. 

We then proceed to obtaining the mean and the variance of the distribution, and finally, we examine the behaviour of this distribution when $n$ tends to infinity: in particular, we show that the normalised prefix exchange distribution converges in distribution to the standard normal distribution. Our result enriches the family of combinatorial sequences which were previously shown to behave asymptotically normally, like the (signless) Stirling numbers of first and second kind, the Eulerian numbers, the adjacent transposition distance distribution and the related distribution of the number of inversions in a permutation --- for precise definitions of these sequences and asymptotic normality results, as well as other examples, see \cite{David, Bender, Flajolet}.

\section{Background and known results}

We recall some basic notions and notation (see e.g. \citet{Bona2012}) that will be useful throughout the text.

\subsection{Permutations and cycles}

For $n \geq 1$, we let $\Sym{n}$ denote the \emph{symmetric group}, i.e., the set of all  permutations of $\{1,2,\ldots,n\}$ together with the usual function composition operation $\circ$ applied from right to left. We view permutations as sequences and denote them using lower case Greek letters, i.e., $\pi=\langle\pi_1\ \pi_2\ \cdots\ \pi_n\rangle$, where  $\pi_i=\pi(i)$ for $1\leq i\leq n$. We will sometimes find it convenient to \emph{reduce} permutations in the following sense.

\begin{definition}
    \cite{DBLP:series/eatcs/Kitaev11} The \emph{reduced form} of a permutation $\sigma$ of a set $\{j_1, j_2, \ldots, j_r\}$ with $j_1<j_2<\cdots<j_r$ is the permutation $red(\sigma)\in\Sym{r}$ obtained by replacing $j_i$ with $i$ in $\sigma$ for all $i$ such that $1\leq i\leq r$.
\end{definition}

As is well-known, every permutation $\pi$ decomposes in a single way into disjoint cycles (up to the ordering of cycles and of elements within each cycle).  For instance, when $\pi=\langle 4\ 1\ 6\ 2\ 5\ 7\ 3\rangle$, the \emph{disjoint cycle decomposition} is $\pi=(1,4,2)(3,6,7)(5)$ (notice the parentheses and the commas). We use $c_1(\pi)$ to denote the number of cycles of length $1$, or \emph{fixed points}, of $\pi$, and $c_{\geq 2}(\pi)$ to denote the number of cycles of length at least $2$ of $\pi$. 

Let $dcd(\pi)$ denote the disjoint cycle decomposition of $\pi$. It will sometimes be convenient to abuse notation by writing, for some permutation $\pi\in\Sym{n}$, $\sigma=dcd(\pi)\cup(n+1)$, to express the fact that the disjoint cycle decomposition of $\pi$ and $\sigma$ differ only by the fixed point $\sigma_{n+1}=n+1$, which does not exist in $\pi$.

Recall that, for $0 \leq k \leq n$, the \emph{signless Stirling number of the first kind} $\firststirlingnumber{n}{k}$ counts the number of permutations of $n$ elements with $k$ cycles, with the convention that $\firststirlingnumber{n}{0}=0$ for $n > 1$ and
$\firststirlingnumber{0}{0}=1$. 
These numbers are well-known to appear in the following series expansion of the \emph{ascending factorial}: 
\begin{equation}\label{eqn:ascending-factorial}
x^{\overline{n}} =  x(x+1)\cdots (x+n-1)= \sum_{k=0}^{n} \firststirlingnumber{n}{k} x^k.
\end{equation}
The \emph{signed Stirling number of the first
kind} is $s(n,k)  =(-1)^{n-k} \firststirlingnumber{n}{k}$.

\subsection{Prefix exchanges}

For every $i=2,3,\ldots,n$, the \emph{prefix exchange} $(1,i)$ applied to a permutation $\pi$ in $\Sym{n}$ transforms $\pi$ into $\pi\circ(1,i)$ by swapping elements $\pi_1$ and $\pi_i$. The \textit{prefix exchange distance} of $\pi$, denoted by $pexc(\pi)$, is the minimum number of prefix exchanges needed to sort the permutation $\pi$, i.e., to transform it into the identity permutation $\iota=\langle 1\ 2\ \cdots\ n\rangle$. 
\citet{akers-star} proved the following formula for computing the prefix exchange
distance:

\begin{theorem}\label{formula_dist}
  \cite{akers-star}
  The prefix exchange distance of $\pi$ in $\Sym{n}$ is equal to
\begin{equation}\label{form}
pexc(\pi)=n+c_{\geq 2}(\pi)-c_1(\pi)-\left\{
    \begin{array}{ll}
      0 & \mbox{if } \pi_1= 1, \\
      2 & \mbox{otherwise}.
    \end{array}
  \right.
\end{equation}
\end{theorem}

\citet{akers-star} referred to the Cayley graph of $\Sym{n}$ generated by prefix exchanges as the ``$n$-star graph''. 
They proved that the diameter of that graph, or equivalently the largest value that the distance can reach, is $\left\lfloor 3(n-1)/2\right\rfloor$.

Let $n \geq 1$ be fixed. For $k \geq 0$, 
we let $W_{n,k}$ denote the number of permutations in $\Sym{n}$ which are at prefix exchange distance $k$ from the identity permutation. These numbers are called in the literature ``Whitney numbers of the second kind for the star poset'' or ``surface areas for the star graph''.
An explicit formula for these numbers was first given by \citet{portier-whitney}, and later corrected by \citet{shen-whitney}: 

\begin{theorem}\label{formula_W}
 \cite{shen-whitney} The Whitney numbers of the second kind for the star poset are given as follows. Let $n\geq 1$ and $k$ such that $0\leq k\leq \left\lfloor 3(n-1)/2 \right\rfloor$ and denote, for $0 \leq i \leq \min(n-1,k+1)$:
$$ T_i = \max\left\{0,\left\lceil\frac{k-2i}{2}\right\rceil\right\}, \  S_i = \min\left\{n-1-i,\left\lfloor\frac{k+1-i}{2}\right\rfloor\right\}.
$$
With these notation, we have:
$$W_{n,k}=\sum_{i=0}^{\min(n-1,k+1)}\sum_{t=T_i}^{S_i}\binom{n-1}{i}\binom{n-1-i}{t}
s(i+1,k-i+1-2t)(-1)^{k+2-t}.$$
\end{theorem}

Using different approaches, \citet{imani-article} and \citet{cheng-whitney} give alternative explicit formulas for $W_{n,k}$. The following recurrence relations are also known (see \citet{portier-whitney} for the first one and \citet{qiu-some} for the second and the third):

\begin{theorem}\label{thm:recurrences}
 The Whitney numbers of the second kind for the star poset obey the following recurrence relations: for $n\geq 1$ and $3\leq k\leq \left\lfloor 3(n-1)/2\right\rfloor$, we have:
\begin{equation} \label{whitney-recurrence-i}
W_{n,k}=W_{n-1,k}+(n-1)W_{n-1,k-1}-(n-2)W_{n-2,k-1}+(n-2)W_{n-2,k-3},
\end{equation}
\begin{equation} \label{whitney-recurrence-ii}
 W_{n, k} = (n-1)W_{n-1, k-1} + \sum_{j=1}^{n-2} j W_{j, k-3},
\end{equation}
with $W_{n,0}=1, W_{n,1}=n-1$ and $W_{n,2}=(n-1)(n-2).$
We also have, for $n\geq 1$ and $0\leq k\leq \left\lfloor\frac{3(n-1)}{2}\right\rfloor$:
\begin{equation} \label{whitney-recurrence-iii}
 W_{n+k+1, k} = \sum_{i=1}^{k+1} (-1)^{i+1} \binom{k+1}{i} W_{n+k+1-i, k}.
\end{equation}
\end{theorem}

\section{Combinatorial derivation of the formula for $W_{n,k}$}

The explicit formula for $W_{n,k}$ given in \Cref{formula_W} was obtained by \citet{portier-whitney} (notwithstanding some errors later corrected by \citet{shen-whitney}) using a generating function technique: they first derived the generating function of these numbers and then used it to deduce a formula for $W_{n,k}$. 
We give here a direct combinatorial derivation of the formula in  \Cref{formula_W} based on \emph{derangements}, i.e., permutations without any fixed point. We proceed in two steps, by first computing the number $W^{(1)}_{n,k}$  of permutations at distance $k$ that fix $1$ and then the number $W^{(2)}_{n,k}$ of permutations at distance $k$ that do not fix $1$. 
We will need the following preliminary result, which counts the number $d(n, k)$ of derangements in $\Sym{n}$ with $k$ cycles.

\begin{lemma}\cite{riordan-intro}\label{lemma:number-of-derangements-with-k-cycles}
For $1 \leq k \leq n$, we have 
$$d(n,k)= \sum_{j=0}^{k} (-1)^j \binom{n}{j} \firststirlingnumber{n-j}{k-j},$$
with the convention $d(n,0)=0$.
\end{lemma}

The following well-known relation (see e.g. \citet[page 167]{DBLP:books/daglib/0068021}) will also be useful:
\begin{equation}\label{eqn:handy-rule}
\binom{r}{m}\binom{m}{p}=\binom{r}{p}\binom{r-p}{m-p}
\end{equation}
for any $m$, $p$, $r$ $\in\mathbb{N}$.

\begin{proposition}\label{prop:perms-at-dist-k-fixing-one}
The number of permutations $\pi$ in $\Sym{n}$ with $pexc(\pi)=k$ and $\pi_1=1$ is
\begin{equation}\label{W1}
W^{(1)}_{n,k}=\sum_{\ell=\max(n-k-1,0)}^{\lfloor (2n-k-2)/2 \rfloor}
\sum_{j=0}^{k-n+\ell+1} \binom{n-1}{\ell+j}\binom{\ell+j}{j}
(-1)^j \firststirlingnumber{n-\ell-j-1}{k-n+\ell-j+1}.
\end{equation}    
\end{proposition}
\begin{proof}
We sum over all possible values $i$ for the number of fixed
points of $\pi$. For a permutation $\pi$ with $pexc(\pi)=k$
and $c_1(\pi)=i$, \Cref{form} implies that 
$c_{\geq 2}(\pi)=k-n+i$. From the conditions $k-n+i \geq 0$
and $n-i \geq k-n+i$

we easily obtain the following bounds for $i$: $\max(n-k,1) \leq i \leq \lfloor (2n-k)/2 \rfloor$.

Since $\pi_1=1$, there are $\binom{n-1}{i-1}$
choices for the other $i-1$ fixed points. The
remaining $n-i$ elements must form $k-n+i$ cycles of length
at least 2, and there are exactly $d(n-i,k-n+i)$ ways to do this. We obtain
\begin{align*}
W^{(1)}_{n,k} &= \sum_{i=\max(n-k,1)}^{\lfloor (2n-k)/2 \rfloor}
\binom{n-1}{i-1} d(n-i,k-n+i)\\
& =\sum_{i=\max(n-k,1)}^{\lfloor (2n-k)/2 \rfloor}
\binom{n-1}{i-1} \sum_{j=0}^{k-n+i} (-1)^j \binom{n-i}{j}
\firststirlingnumber{n-i-j}{k-n+i-j} & \mbox{(using \Cref{lemma:number-of-derangements-with-k-cycles})}.
\end{align*}
Setting $\ell=i-1$, we have
$$W^{(1)}_{n,k}=\sum_{\ell=\max(n-k-1,0)}^{\lfloor (2n-k-2)/2 \rfloor}
\binom{n-1}{\ell} \sum_{j=0}^{k-n+\ell+1} (-1)^j
\binom{n-\ell-1}{j}
\firststirlingnumber{n-\ell-j-1}{k-n+\ell-j+1},$$
and using \Cref{eqn:handy-rule} 
with $r=n-1$, $m=\ell+j$ and $p=j$ 
allows us to complete the proof.

\end{proof}

\begin{proposition}\label{prop:perms-at-dist-k-not-fixing-one}
The number of permutations $\pi$ in $\Sym{n}$ with $pexc(\pi)=k$ and $\pi_1\neq 1$ is
\begin{eqnarray*}
W^{(2)}_{n,k} &=&
\sum_{i=\max(n-k-2,0)}^{\lfloor(2n-k-2)/2\rfloor}
\sum_{j=0}^{k-n+i+2} \binom{n-1}{i+j}\binom{i+j}{j} (-1)^j
\firststirlingnumber{n-i-j}{k-n+i+2-j} \\
& + &\sum_{i=\max(n-k-1,0)}^{\lfloor(2n-k-2)/2\rfloor}
\sum_{j=1}^{k-n+i+2} \binom{n-1}{i+j-1}\binom{i+j-1}{j-1} (-1)^j
\firststirlingnumber{n-i-j}{k-n+i+2-j}.
\end{eqnarray*}

\end{proposition}
\begin{proof}
As in the proof of \Cref{prop:perms-at-dist-k-fixing-one}, we sum over all possible values $i$ for the number
of fixed points of $\pi$. In this case, if a permutation $\pi$ has
$i$ fixed points and is at distance $k$, then \Cref{form} implies
that $c_{\geq 2}(\pi)=k-n+i+2$.  
From the two conditions: $k-n+i+2 \geq 0$ and $n-i \geq k-n+i+2$, we derive the following bounds for $i$: $\max(n-k-2,0) \leq i \leq \lfloor(2n-k-2)/2\rfloor.$

Since $\pi_1\neq 1$, we have $\binom{n-1}{i}$ choices for the
$i$ fixed points. Furthermore, the remaining $n-i$ elements must
form $k-n+i+2$ cycles of length at least 2. We obtain:
\begin{align*}
W^{(2)}_{n,k} &= \sum_{i=\max(n-k-2,0)}^{\lfloor(2n-k-2)/2\rfloor}
\binom{n-1}{i} d(n-i,k-n+i+2)\\
& =\sum_{i=\max(n-k-2,0)}^{\lfloor(2n-k-2)/2\rfloor}
\binom{n-1}{i} \sum_{j=0}^{k-n+i+2} (-1)^j \binom{n-i}{j}
\firststirlingnumber{n-i-j}{k-n+i+2-j}& \mbox{(using \Cref{lemma:number-of-derangements-with-k-cycles})}.
\end{align*}
One can easily check that the following relations hold:
\begin{align*}
\binom{n-1}{i}\binom{n-i}{j} &=\binom{n-1}{i+j}\binom{i+j}{j}\frac{n-i}{n-i-j} & \mbox{(using \Cref{eqn:handy-rule})}\\
&= \binom{n-1}{i+j}\binom{i+j}{j}\left( 1+
\frac{j}{n-i-j}\right)\\
&= \binom{n-1}{i+j}\binom{i+j}{j} +
\binom{n-1}{i+j-1}\binom{i+j-1}{j-1},
\end{align*}
where for the last line we have used the fact that
$$\binom{n-1}{i+j}\binom{i+j}{j}\frac{j}{n-i-j}=\binom{n-1}{i+j-1}\binom{i+j-1}{j-1}.$$
This allows us to obtain the formula in the statement, and the proof is complete.
\end{proof}

\Cref{prop:perms-at-dist-k-fixing-one,prop:perms-at-dist-k-not-fixing-one} allow us to recover the expression in \Cref{formula_W} as follows. First, decompose the expression in \Cref{prop:perms-at-dist-k-not-fixing-one} into $S_1$ and $S_2$:
\begin{eqnarray*}
W^{(2)}_{n,k} &=&
\overbrace{\sum_{i=\max(n-k-2,0)}^{\lfloor(2n-k-2)/2\rfloor}
\sum_{j=0}^{k-n+i+2} \binom{n-1}{i+j}\binom{i+j}{j} (-1)^j
\firststirlingnumber{n-i-j}{k-n+i+2-j}}^{S_1} \\
& + &\underbrace{\sum_{i=\max(n-k-1,0)}^{\lfloor(2n-k-2)/2\rfloor}
\sum_{j=1}^{k-n+i+2} \binom{n-1}{i+j-1}\binom{i+j-1}{j-1} (-1)^j
\firststirlingnumber{n-i-j}{k-n+i+2-j}}_{S_2}, 
\end{eqnarray*}
and set $u=j-1$ in $S_2$ to obtain
$$S_2= - \sum_{i=\max(n-k-1,0)}^{\lfloor(2n-k-2)/2\rfloor}
\sum_{u=0}^{k-n+i+1} \binom{n-1}{i+u}\binom{i+u}{u} (-1)^u
\firststirlingnumber{n-i-u-1}{k-n+i-u+1}.$$ 
Using \Cref{W1}, we note that 
$S_2=- W^{(1)}_{n,k}$, so 
$W_{n,k}= W^{(1)}_{n,k}+ W^{(2)}_{n,k} = S_1$. If we then set $\ell=n-i-j-1$ in $S_1$,
we obtain
$$W_{n,k}= \sum_{\ell=0}^{\min(n-1,k+1)}
\sum_{j=\max(\lceil (k-2\ell)/2\rceil,0)}^{\min(n-1-\ell,\lfloor
(k+1-\ell)/2 \rfloor)} \binom{n-1}{\ell}\binom{n-1-\ell}{j} (-1)^j
\firststirlingnumber{\ell+1}{k-\ell-2j+1},$$ 
and the fact that $s(n,k)=(-1)^{n-k} \firststirlingnumber{n}{k}$ yields the
formula in \Cref{formula_W}.

\section{Combinatorial proof of the recurrence relations}

We now turn to the recurrence relations in \Cref{thm:recurrences}. We will find it convenient to introduce the following additional notation:
\begin{align*}
\Symdist{n}{k}&=\{\pi\in\Sym{n}\ |\ pexc(\pi)=k\} \mbox{ (so $W_{n,k} = |\Symdist{n}{k}|$); and}\\
\Sym{n, k, i}&=\{\pi\in\Symdist{n}{k}\ |\ \pi_i=i\}. 
\end{align*}

\subsection{Proof of \Cref{whitney-recurrence-i}}

\citet{portier-whitney} prove the recurrence relation in \Cref{whitney-recurrence-i} using a generating function technique.
We give here a direct combinatorial proof, again distinguishing between permutations that fix the first element and those that do not.

\begin{proof}
Let  $n\geq 1$ and $k$ such that $3\leq k\leq \left\lfloor 3(n-1) / 2 \right\rfloor$ be fixed. 

\begin{enumerate}
    \item {\bf permutations $\pi$ in $\Symdist{n}{k}$ with $\pi_1\neq 1$:} we compute $W^{(2)}_{n,k}$ by summing over all permutations $\pi$ which are at distance $k$ and verify $\pi_1=i$ for a given $i \in\{2,3,\ldots, n\}$.

For a given $i$ verifying  $2 \leq i \leq n$, we introduce the following mappings:
\begin{align*}
    &\phi_{i}:\{\pi\in\Symdist{n}{k}\ |\ \pi_1=i\}\rightarrow \Sym{n,k-1,i}:&&\pi\mapsto\sigma=\pi\circ(1,i),\\
    &\psi_{i}:\Sym{n,k-1,i}\rightarrow\Symdist{n-1}{k-1}:                   &&\sigma\mapsto\tau=red(dcd(\sigma)\setminus (i)).
\end{align*}

Both mappings are bijective and allow us to associate any element $\pi \in \Symdist{n}{k}$ with $\pi_1=i$ to an element $\tau=\psi_{i}(\phi_{i}(\pi)) \in \Symdist{n-1}{k-1}$. Therefore,

\begin{equation}\label{cardinal}
\left| \{\pi\in\Symdist{n}{k}\ |\ \pi_1=i\} \right| = W_{n-1,k-1}.
\end{equation}
Since this holds for every $i$ such that $2 \leq i \leq n$, we obtain
\begin{equation}\label{cardinal2}
W^{(2)}_{n,k}= \left| \{\pi\in\Symdist{n}{k}\ |\ \pi_1 \neq 1\} \right| = (n-1)W_{n-1,k-1},
\end{equation}
which in turn yields

\begin{equation}\label{relation_W-H}
W_{n,k}=W^{(1)}_{n,k}+W^{(2)}_{n,k}=W^{(1)}_{n,k}+ (n-1)W_{n-1,k-1}.
\end{equation}

\item {\bf permutations $\pi$ in $\Symdist{n}{k}$ with $\pi_1=1$:} in order to compute $W^{(1)}_{n,k}$, we further distinguish permutations in 

$\Sym{n,k,1}$ 
based on the value of their last element. More precisely, for $i=2,3,\ldots,n$, let $W^{(1,i)}_{n,k}$  denote the number of permutations 

$\pi$ in $\Sym{n,k,1}$ with $\pi_n=i$. 
We have 
$$W^{(1)}_{n,k}=W^{(1,n)}_{n,k}+ \sum_{i=2}^{n-1} W^{(1,i)}_{n,k}.$$

We first note that $W^{(1,n)}_{n,k}=W^{(1)}_{n-1,k}$, since any permutation 

$\pi \in \Sym{n,k,1}$ with $\pi_n=n$ 
can be bijectively mapped onto a permutation 

$\tau \in\Sym{n-1,k,1}$
by deleting  $\pi_n=n$.

For $i\in\{2,3,\ldots,n-1\}$,  we will next compute $W^{(1,i)}_{n,k}$. Let  

$\pi \in \Sym{n,k,1}$ be a permutation with 
$\pi_n=i$. Then deleting $\pi_1=1$ and renaming element $n$ into $1$

maps $\pi$ bijectively onto a permutation $\tau \in \Sym{n-1}$ with $\tau_1=i$ and having the same cycle structure as $\pi$ except for the deleted singleton $(1)$. Using \Cref{form}, we can easily see that $pexc(\tau)=k-2$, so $\tau\in \Symdist{n-1}{k-2}$ and \Cref{cardinal} implies that the number of such permutations $\tau$ equals $W_{n-2,k-3}$.
Therefore, the number of permutations $\pi$ in $\Sym{n,k,1}$ with $\pi_n=i$ is $(n-2)W_{n-2,k-3}$. 
\end{enumerate}
From the above discussion, we deduce 

\begin{equation}\label{relation_W1}
W^{(1)}_{n,k}  =  W^{(1)}_{n-1,k} + (n-2)W_{n-2,k-3}.
\end{equation}
Using \Cref{relation_W-H} for $W^{(1)}_{n-1,k}$ we further obtain 
$$W^{(1)}_{n,k}=W_{n-1,k}-(n-2)W_{n-2,k-1} +(n-2)W_{n-2,k-3},$$

from which we finally recover \Cref{whitney-recurrence-i} by replacing the left-hand side using again \Cref{relation_W-H}.

\end{proof}

\subsection{Proof of \Cref{whitney-recurrence-ii}}

\begin{proof}
Let again $n\geq 1$ and $k$ such that $3\leq k\leq \left\lfloor\frac{3(n-1)}{2}\right\rfloor$ be fixed.
With the same notation as in the previous subsection, and using \Cref{relation_W-H}, we see that it suffices to prove
\begin{equation} \label{form_W1}
W^{(1)}_{n,k} = \sum_{i=1}^{n-2} i W_{i, k-3}.
\end{equation}

For $1 \leq i \leq n-2$, let $\mathcal{F}_i$ denote the set of permutations 

$\pi \in \Sym{n,k,1}$ with $i+2=\argmax_{1\leq j\leq n} \{\pi_j\neq j\}$. 
We thus have $\pi_{i+2}\neq i+2$ and $\pi$ fixes all elements from $i+3$ to $n$.
Note that $\max\{j=1,2,\ldots,n: \pi_j \neq j\}\not\in\{1,2\}$ since we assume $k \geq 3$.
Therefore,
\begin{equation}\label{decomposition}
W^{(1)}_{n,k} = \sum_{i=1}^{n-2} |\mathcal{F}_i|.
\end{equation}

To any permutation $\pi \in \mathcal{F}_i$, we can bijectively associate a permutation $\tau \in \Sym{i+1}$ obtained from $\pi$ by deleting singletons $(1), (i+3),(i+4),\ldots,(n)$ and renaming element $i+2$ into $1$. The resulting permutation $\tau$ verifies $\tau_1 \neq 1$, and $pexc(\tau)=k-2$ by \Cref{form}. Using this bijection and \Cref{cardinal2}, we obtain
$$ |\mathcal{F}_i| = W_{i+1,k-2}-W^{(1)}_{i+1,k-2} = iW_{i,k-3},$$
and \Cref{decomposition} allows us to complete the proof.
\end{proof}

We note that \citet{qiu-some} gave an alternative combinatorial proof for \Cref{relation_W1} and then derived \Cref{form_W1} by  recurrence.

\subsection{Proof of \Cref{whitney-recurrence-iii}}

We give here a combinatorial proof for  \Cref{whitney-recurrence-iii}, which was proved by \citet{qiu-some} by induction, in a direct computational manner.

\begin{proof}
For every $i$ such that $1 \leq i \leq k+1$, we let $B_i=\Sym{n+k+1,k,n+i}$.
We first prove that
\begin{equation}\label{union}
W_{n+k+1,k} = \left| B_1 \cup B_2 \cup \cdots \cup B_{k+1} \right|.
\end{equation}
To achieve this, we show that any 
permutation 

$\pi \in \Symdist{n+k+1}{k}$ 
fixes at least one element among $n+1,n+2,\ldots,n+k+1$, and therefore $\pi \in  \bigcup_{i=1}^{k+1} B_{i} $, which will imply \Cref{union}.  
\\
\begin{itemize}
    \item 
 If $\pi_1=1$, from \Cref{form} we deduce
$c_1(\pi) = n+1 + c_{\geq 2}(\pi) \geq n+1.$ 
Therefore, at least one element among $n+1,n+2,\ldots,n+k+1$ must be a singleton. \\
\item If $\pi_1 \neq 1$, then \Cref{form} implies $c_1(\pi) = n+c_{\geq 2}(\pi)-1 \geq n.$ 
Since 1 is not a singleton, there must also be at least one  singleton among the elements  $n+1,n+2,\ldots,n+k+1$. 
\end{itemize}

Since the roles of the elements $n+1,\ldots,n+k+1$ are perfectly interchangeable, we have $|B_{j_1} \cap B_{j_2}\cap \cdots \cap B_{j_i} | = |B_1 \cap B_2\cap \cdots \cap B_i |, $ for every $j_1, j_2, \ldots, j_i$ such that $1\leq j_1 < j_2 < \cdots < j_i \leq n$ and every $i$ such that $1\leq i\leq k+1.$
From \Cref{union} and the inclusion-exclusion rule, we deduce
$$
W_{n+k+1,k} = \sum_{i=1}^{k+1} (-1)^{i+1} \binom{k+1}{i} |B_1 \cap B_2\cap \cdots \cap B_i |.
$$

In order to prove \Cref{whitney-recurrence-iii}, it must be noted that for every $i$ such that $1 \leq i \leq k+1$: 
\begin{equation}\label{intersection}
|B_1 \cap B_2\cap \cdots \cap B_i  | = W_{n+k+1-i,k}.
\end{equation}

Indeed, for every $i$ such that $1 \leq i \leq k+1$, we can define the following bijection,
\begin{align*}
\xi_i : &\ B_1 \cap B_2\cap \cdots \cap B_i \rightarrow \Symdist{n+k+1-i}{k} \\
 : &\ \pi\mapsto\tau=red(dcd(\pi)\setminus \{(n+1),\ldots,(n+i)\}),
\end{align*}
which proves \Cref{intersection}.    
\end{proof}

\section{Generating function, mean and variance of the distance distribution}

We obtain in this section expressions for the mean $\mu_n$ and the variance $\sigma_n^2$ of the prefix exchange distance distribution. More precisely, for a uniform random permutation $\pi$ in $\Sym{n}$, we have $\mathbb{P}(pexc(\pi)=k)=W_{n,k}/n!$ and
\begin{align*}
\mu_n &=\mathbb{E}(pexc(\pi)) = \frac{1}{n!}\sum_{k=0}^{\infty} kW_{n,k},\\
\sigma_n^2 &= \mathrm{Var}(pexc(\pi))= \frac{1}{n!}\sum_{k=0}^{\infty}k^2 W_{n,k} - \mu_n^2.
\end{align*}

We start by computing the ordinary generating function
$$W_n(x) = \sum_{k=0}^\infty W_{n,k}x^k.$$
As is well-known (see e.g. \citet{wilf}), the mean and the variance can be obtained by derivating $W_n(x)$:
\begin{align} 
\mu_n &= \frac{W_n'(1)}{n!};\label{mean} \\
\sigma_n^2 &=  \frac{W_n''(1)}{n!} + \mu_n - \mu_n^2 = \frac{W_n''(1)}{n!} - \mu_n(\mu_n - 1).\label{variance}
\end{align}

\subsection{The generating function}

We give here an alternative proof of a formula known to \citet{portier-whitney} for computing the ordinary generating function $W_n(x)$ of the sequence $(W_{n,k})_{k \geq 0}$.

Our proof uses \Cref{formula_W} as a starting point, whereas \citeauthor{portier-whitney} (with subsequent corrections by \citet{shen-whitney}) first computed the generating function, then used it to derive the expression in \Cref{formula_W}.

\begin{theorem}
The ordinary generating function for the prefix exchange distance distribution is given, for every $x \in \mathbb{C}$, by the following formula:
\begin{equation}\label{gen_fct}
W_n(x) = \sum_{i=0}^{n-1} \binom{n-1}{i} (1-x^2)^{n-1-i}  x^{i} \prod_{j=1}^i (x+j).
\end{equation}
\end{theorem}
\begin{proof}
Let $n \geq 1$ be fixed. Interchanging the order of summation in the formula of \Cref{formula_W} yields
\begin{equation*}
W_n(x) = \sum_{i=0}^{n-1} \binom{n-1}{i} \sum_{t=0}^{n-1-i}
\binom{n-1-i}{t} \sum_{k=2t+i-1}^{2t+2i}
s(i+1,k-i+1-2t)(-1)^{k+2-t} x^k,\\
\end{equation*}
with the convention $s(1,-1)=0$.
Given $i$ and $t$ such that $0 \leq i \leq n-1$ and $0 \leq t \leq n-i-1$, the bounds on $k$ come from the conditions $\lceil\frac{k-2i}{2}\rceil  \leq t \leq \lfloor\frac{k+1-i}{2}\rfloor$ appearing in \Cref{formula_W}.
Setting $j = k - 2t -i+1$, we obtain
\begin{align*}
& W_n(x) \\
=& \sum_{i=0}^{n-1} \binom{n-1}{i} \sum_{t=0}^{n-1-i}
\binom{n-1-i}{t} (-1)^t x^{2t+i-1} \sum_{j=0}^{i+1}
s(i+1,j)(-1)^{i+1+j} x^j\\
=&\sum_{i=0}^{n-1} \binom{n-1}{i} \sum_{t=0}^{n-1-i}
\binom{n-1-i}{t} (-1)^t x^{2t+i-1} \sum_{j=0}^{i+1}
\firststirlingnumber{i+1}{j} x^j\\
=&\sum_{i=0}^{n-1} \binom{n-1}{i} \sum_{t=0}^{n-1-i}
\binom{n-1-i}{t} (-1)^t x^{2t+i-1} x^{\overline{i+1}}\quad\quad\quad\quad \mbox{(using \Cref{eqn:ascending-factorial})}\\
= &\sum_{i=0}^{n-1} \binom{n-1}{i} x^{i-1}
 x^{\overline{i+1}} \sum_{t=0}^{n-1-i} \binom{n-1-i}{t} (-1)^t
 x^{2t}.
\end{align*}
The expression in the statement then follows from Newton's binomial formula, 
with the convention that $\prod_{j=1}^0 (x+j) = 1$.
\end{proof}

\subsection{Mean and variance of the distance distribution}

Let $\pi$ be a uniform random permutation in $\Sym{n}$. We will derive expressions for its mean $\mu_n$ and variance $\sigma_n^2$

which will involve the $n$-th harmonic number $H_n = \sum_{k=1}^n 1/k$.

Let $n \geq 3$. Using \Cref{gen_fct}, we can write:
\begin{equation}\label{sum}
W_n(x) = \prod_{j=1}^{n-1} [x(x+j)] + (n-1)(1-x^2)  \prod_{j=1}^{n-2} [x(x+j)] + g(x) +  (1-x^2)^3 h(x),
\end{equation}
where $h(x)$ is some polynomial function and
$$g(x) =  \frac{(n-1)(n-2)}{2}  (1-x^2)^2  \prod_{j=1}^{n-3} [x(x+j)] .$$

\subsubsection{Computation of the mean}

We now derive an expression for the expected prefix exchange distance. We note that the value of $\mu_n$ can be obtained as a particular case of Theorem 6.1, page~203 of \citet{cheng-average} by setting $k=n-1$ in the formula they derive. We give here a direct proof of that expression, which provides elements that will prove useful in obtaining the variance of the prefix exchange distance.

\begin{theorem}
Let $n \geq 1$. The expected value $\mu_n$ of the prefix exchange distance for a uniform random permutation in $\Sym{n}$ equals
\begin{equation}\label{formula_mean}
\mu_n = n + H_n - 4 + \frac{2}{n}.
\end{equation}
\end{theorem}

\begin{proof}
We evaluate the derivative of $W_n(x)$ at $x=1$ using the simplified expression in \Cref{sum}. 
For $n \geq 3$, we have
\begin{align*}
W_n'(x) &=   \sum_{i=1}^{n-1}  \frac{2x+i}{x(x+i)}  \prod_{j=1}^{n-1} [x(x+j)] +  (n-1)\left(-2x + (1-x^2) \sum_{i=1}^{n-2}  \frac{2x+i}{x(x+i)} \right)\prod_{j=1}^{n-2}[x (x+j)] \\
&\ \ \ + g'(x)+  [(1-x^2)^3 h(x)]'.
\end{align*}
When $x=1$, both $g'(x)$ and $[(1-x^2)^3 h(x)]'$ vanish, and we obtain:
\begin{align*}
W_n'(1)&= n! \sum_{i=1}^{n-1}  \frac{i+2}{i+1}  - 2(n-1)(n-1)!\\
& =  n!\left(n-1 + \sum_{i=1}^{n-1}  \frac{1}{i+1}\right) - 2(n-1)(n-1)!\\
& = n!(n+ H_n - 2) - 2(n-1)(n-1)!.
\end{align*}
Using \Cref{mean}, we deduce
$$
\mu_n = n + H_n - 2 - 2\left(1-\frac{1}{n}\right) = n + H_n - 4 + \frac{2}{n} .
$$
Note that this expression remains valid for $n=1$ and $n=2$; the above assumption $n\geq 3$ was forced on us by \Cref{sum}.

\end{proof}

\subsubsection{Computation of the variance}

We will prove the following:

\begin{theorem}\label{thm:variance}
Let $n \geq 2$. The variance $\sigma_n^2$ of the prefix exchange distance for a uniform random permutation in $\Sym{n}$ equals
\begin{equation}\label{formula_var}
\sigma_n^2 =  H_n + \frac{4}{n} - \frac{8}{n^2}  - \sum_{j=1}^n \frac{1}{j^2}.
\end{equation}
\end{theorem}

\begin{proof}
We evaluate the second derivative of $W_n(x)$ at $x=1$.
We first note that we can rewrite the previous expression for $W_n'(x)$ as
\begin{align*}
W_n'(x) &=  \left\{ \sum_{i=1}^{n-1}  \frac{2x+i}{x(x+i)} -\frac{2(n-1)}{x+n-1} \right\} \prod_{j=1}^{n-1}[x (x+j)]\\
& \ \ \ +(n-1)(1-x^2)\sum_{i=1}^{n-2}  \frac{2x+i}{x(x+i)} \prod_{j=1}^{n-2}[x (x+j)] + g'(x)+  [(1-x^2)^3 h(x)]'.
\end{align*}
Using the fact that $$\frac{2x+i}{x(x+i)} = \frac{1}{x} + \frac{1}{x+i}$$ and derivating a second time, we obtain
\begin{align*}
W_n''(x) &=  \left\{ -\sum_{i=1}^{n-1} \left(\frac{1}{x^2}+\frac{1}{(x+i)^2}  \right)+\frac{2(n-1)}{(x+n-1)^2} \right\} \prod_{j=1}^{n-1} [x(x+j)] \\
& \ \ \ +  \left\{ \sum_{i=1}^{n-1}  \left( \frac{1}{x} + \frac{1}{x+i} \right) -\frac{2(n-1)}{x+n-1} \right\} \sum_{k=1}^{n-1}   \left( \frac{1}{x} + \frac{1}{x+k} \right) \prod_{j=1}^{n-1} [x(x+j)]\\
& \ \ \ -2(n-1)x  \sum_{i=1}^{n-2}  \left( \frac{1}{x} + \frac{1}{x+i} \right) \prod_{j=1}^{n-2}[x (x+j)] + (1-x^2)u(x)\\
& \ \ \ + g''(x) +  [(1-x^2)^3 h(x)]'',
\end{align*}
where $u(x)$ is some polynomial function.
Since $ [(1-x^2)^3 h(x)]''$ vanishes when $x=1$, we have
\begin{align*}
W_n''(1) 
&=  \left\{ -\sum_{i=1}^{n-1} \left(1+\frac{1}{(1+i)^2}  \right)+\frac{2(n-1)}{n^2} \right\} \prod_{j=1}^{n-1} (1+j) \\
& \ \ \ +  \left\{ \sum_{i=1}^{n-1}  \left(1 + \frac{1}{1+i} \right) -\frac{2(n-1)}{n} \right\} \sum_{k=1}^{n-1}   \left( 1 + \frac{1}{1+k} \right) \prod_{j=1}^{n-1} (1+j)\\
& \ \ \ -2(n-1)  \sum_{i=1}^{n-2}  \left( 1 + \frac{1}{1+i} \right) \prod_{j=1}^{n-2}(1+j).
\end{align*}
Replacing $\prod_{j=1}^{n-1}(1+j)$ with $n!$ and $\sum_{i=1}^{n-1}  \left(1 + \frac{1}{1+i}\right)$ with $n+H_n-2$ 
yields
\begin{align*}
W_n''(1) &= n! \left\{- n + 2  + \frac{2}{n} - \frac{2}{n^2} - \sum_{j=1}^n \frac{1}{j^2} + \left(n + H_n - 4  + \frac{2}{n}\right) (n+H_n-2) \right\}\\
& \ \ \ - 2(n-1)(n-1)!\left(n+H_n-3-\frac{1}{n}\right)+ g''(1).
\end{align*}

We must now compute $g''(1)$. We have, with $f(x)$ being  some polynomial function:
$$g'(x)=2(n-1)(n-2)\left\{ -(1-x^2)x^{n-2}  \prod_{j=1}^{n-3} (x+j) + (1-x^2)^2 f(x)\right\}.$$
When derivating a second time and taking $x=1$ we obtain
$$g''(1)=4(n-2)(n-1)!.$$
Injecting this expression in the previous formula for $W_n''(1)$ and dividing by $n!$ gives
\begin{align*}
\frac{W''_n(1)}{n!}&= - n + 2  + \frac{2}{n} - \frac{2}{n^2} - \sum_{j=1}^n \frac{1}{j^2} + \left(n + H_n - 4  + \frac{2}{n}\right) \left(n+H_n-2 \right)\\
& \ \ \ - 2\left(1-\frac{1}{n}\right) \left(n+H_n-3-\frac{1}{n}\right)+4-\frac{8}{n}.
\end{align*}

Using \Cref{variance}, we obtain that the variance of the distance distribution equals
\begin{align*}
\sigma_n^2 &=  - n + 6  - \frac{6}{n} - \frac{2}{n^2} - \sum_{j=1}^n \frac{1}{j^2} + \left(n + H_n - 4  + \frac{2}{n}\right) \left(n+H_n-2 \right)\\
& \ \ \ - 2\left(1-\frac{1}{n}\right) \left(n+H_n-3-\frac{1}{n}\right) -  \left(n + H_n - 4 + \frac{2}{n}\right) \left(n + H_n - 5 +\frac{2}{n}\right)\\
& =  - n + 6  - \frac{6}{n} - \frac{2}{n^2} - \sum_{j=1}^n \frac{1}{j^2} + \left(n + H_n - 4  + \frac{2}{n}\right) \left(3-\frac{2}{n}\right)\\
& \ \ \ - 2\left(1-\frac{1}{n}\right) \left(n+H_n-3-\frac{1}{n}\right),
\end{align*}
which finally gives the formula in \Cref{formula_var}.
\end{proof}

\section{Asymptotic behaviour of the distance distribution}

We now study asymptotic properties of the prefix exchange distance distribution, namely, the value of its mean and its variance as well as its convergence as $n\to\infty$.

\begin{proposition}\label{asymptotics}
We have the following asymptotics for the mean and the variance of the prefix exchange distribution when $n$ is large:
\begin{eqnarray*}
\mu_n &=& n + \log n + \gamma - 4 +o(1);\\
\sigma_n^2 & =&  \log n + \gamma - \frac{\pi^2}{6}+o(1),
\end{eqnarray*}
where $\gamma \approx 0.577$ is the Euler-Mascheroni constant and $o(1)$ denotes a sequence converging to 0 as $n \to \infty$.
\end{proposition}
\begin{proof}
Immediate from \Cref{formula_mean,formula_var}, using the well-known results (see, e.g., \citet{DBLP:books/daglib/0068021}):
$$H_n - \log n \longrightarrow \gamma \text{ when } n \to \infty$$
and
$$\sum_{n=1}^{\infty} \frac{1}{n^2} = \frac{\pi^2}{6}.$$
\end{proof}

We further show that, for large $n$, the distribution of the prefix exchange distance for a uniform random permutation $\pi \in \Sym{n}$ is approximately normal, with mean $\mu_n$ and variance $\sigma_n^2$. More precisely, we prove the following:
\begin{theorem}\label{normal}
The normalised prefix exchange distance 
for a uniform random permutation $\pi \in \Sym{n}$, i.e.,
$$
D_n =  \frac{pexc(\pi)- \mu_ n}{\sigma_n}
$$
converges in distribution, when $n \to \infty$, to the standard normal distribution $\mathcal{N}(0,1)$, which means that
$$\mathbb{P}(a < D_n < b) \longrightarrow \frac{1}{\sqrt{2\pi}}\int_{a}^{b} e^{-x^2/2} dx$$
when $n \to \infty$, for every real numbers $a < b$.
\end{theorem}

\begin{remark}
Using the asymptotics for $\mu_n$ and $\sigma_n^2$ derived in \Cref{asymptotics}, the above convergence is equivalent to the following convergence in distribution
$$\frac{pexc(\pi)- n-\log n}{\sqrt{\log n}}   \longrightarrow \mathcal{N}(0,1) \text{ when } n \to \infty,$$
which means that the distribution of the prefix exchange distance for a uniform random permutation $\pi \in \Sym{n}$ is asymptotically normal, with mean $n+\log n$ and variance $\log n$.
\end{remark}

\begin{proof}[\textit{Proof of \Cref{normal}}.]
We will show that the sequence of characteristic functions of the random variables $(D_n)_{n \geq 1}$ converges pointwise, when $n \to \infty$, to the characteristic function of the standard normal distribution, given by $\varphi(t)=e^{-t^2/2}$.
L\'evy's convergence theorem (see e.g. \citet{Billingsley}) 

will then imply that the sequence $(D_n)_{n \geq 1}$ converges in distribution to the standard normal distribution $\mathcal{N}(0,1)$.

Let $\varphi_n$ denote the characteristic function of the random variable $D_n$, defined for $t \in \mathbb{R}$ by $\varphi_n(t)=\mathbb{E}(e^{itD_n})$.
We have

$$\varphi_n(t)= e^{-\frac{it\mu_n}{\sigma_n}}   \sum_{k=0}^{\infty} e^{\frac{itk}{\sigma_n}} \mathbb{P}(pexc(\pi)=k).$$
Since $\pi$ in chosen uniformly at random in $\Sym{n}$, we have $$ \mathbb{P}(pexc(\pi)=k)= \frac{W_{n,k}}{n!},$$ which yields
\begin{equation}\label{phi}
\varphi_n(t) 
= e^{-\frac{it\mu_n}{\sigma_n}}  \frac{W_n(e^{\frac{it}{\sigma_n}})}{n!},
\end{equation}
where $W_n(\cdot)$ is the generating function obtained in \Cref{gen_fct}.
For every $x \in \mathbb{C}$, \Cref{gen_fct} reads
\begin{equation}\label{Wn(x)}
W_n(x) = \sum_{k=0}^{n-1} \frac{(n-1)!}{(n-k-1)!k!} (1-x^2)^{n-1-k} x^k \prod_{j=1}^k (x+j).
\end{equation}

\noindent \Cref{phi,Wn(x)} then yield, for any $t \in \mathbb{R}$:
$$ 
\varphi_n(t) = e^{-\frac{it\mu_n}{\sigma_n}} \frac{1}{n} \sum_{k=0}^{n-1}
e^{\frac{itk}{\sigma_n}}  \frac{(1-e^{\frac{2it}{\sigma_n}})^{n-k-1}}{(n-k-1)!}   \frac{ \prod_{j=1}^k (e^{\frac{it}{\sigma_n}}+j)}{k!}.
$$

We will show that the dominant term is obtained for $k=n-1$ and converges to $e^{-t^2/2}$ when $n \to \infty$, while all other terms vanish at the limit. To that end, let us isolate in $\varphi_n(t)$ the term obtained for $k=n-1$ (which we denote $A_n$) and let $R_n$ denote the sum of all other terms; we obtain:
\begin{equation}\label{formula_phi}
\varphi_n(t) = A_n + R_n,
\end{equation}
with
\begin{equation}\label{A_n}
A_n = \exp\left(-\frac{it(\mu_n - n+1)}{\sigma_n}\right) \frac{ \prod_{j=1}^{n-1} (e^{\frac{it}{\sigma_n}}+j)}{n!}
\end{equation}
and
$$
|R_n| \leq  \frac{1}{n} \sum_{k=0}^{n-2}
 \frac{|1-e^{\frac{2it}{\sigma_n}}|^{n-k-1}}{(n-k-1)!}   \frac{ \prod_{j=1}^k (1+j)}{k!} = \sum_{k=0}^{n-2} \frac{k+1}{n} 
 \frac{|1-e^{\frac{2it}{\sigma_n}}|^{n-k-1}}{(n-k-1)!},
$$
using the fact that $|e^{ix}| = 1$, for  $x \in \mathbb{R}$.

Let us first show that $R_n$ converges to $0$ when $n \to \infty$. Setting $j=n-k-1$ in the above inequality, we obtain:
$$
|R_n| \leq  \sum_{j=1}^{n-1} 
 \frac{|1-e^{\frac{2it}{\sigma_n}}|^j}{j!}  \leq   \sum_{j=1}^{\infty} 
 \frac{|1-e^{\frac{2it}{\sigma_n}}|^j}{j!} =
 \exp( |1-e^{\frac{2it}{\sigma_n}}|) - 1,
$$ 
using the MacLaurin series, and therefore $R_n \longrightarrow 0$ as $n \to \infty$.

To show that $\varphi_n(t) \longrightarrow e^{-t^2/2}$, in light of \Cref{formula_phi}, we must check that
$$ 
A_n= \exp\left(-\frac{it(\mu_n - n+1)}{\sigma_n}\right)  \frac{ \prod_{j=1}^{n-1} (e^{\frac{it}{\sigma_n}}+j)}{n!} \longrightarrow  e^{-t^2/2}.
$$

Note that the product $\prod_{j=1}^{n-1} (e^{\frac{it}{\sigma_n}}+j)$ can be written as a ratio of two Gamma functions.
We recall that for $z \in \mathbb{C}$ with $\mathrm{Re}(z)>0$, the Gamma function is defined as
$$\Gamma(z) = \int_{0} ^{\infty} x^{z-1} e^{-x} dx.$$
By integration by parts, it is easy to see that the Gamma function satisfies the recurrence relation $\Gamma(z+1)=z\Gamma(z)$, which implies, in particular, that for $n \in \mathbb{N}^*$ we have $\Gamma(n) = (n-1)!$.

The same recurrence relation allows us to write: 
$$ 
\prod_{j=1}^{n-1} (e^{\frac{it}{\sigma_n}}+j) = \frac{\Gamma(n+e^{\frac{it}{\sigma_n}})}{\Gamma(e^{\frac{it}{\sigma_n}})},$$
for $n$ sufficiently large to have $\mathrm{Re}(e^{\frac{it}{\sigma_n}} )> 0$. 

We further use the following asymptotic approximation (see e.g. \citet{tricomi}):
$$
\frac{\Gamma(n+e^{ix})}{n!} = n^{e^{ix}-1}(1 + o(1)),
$$
for $x \in \mathbb{R}$ and $n =2, 3, \ldots$,
to deduce
$$
 \frac{ \prod_{j=1}^{n-1} (e^{\frac{it}{\sigma_n}}+j)}{n!} =  \frac{\Gamma(n+e^{\frac{it}{\sigma_n}})}{n!\Gamma(e^{\frac{it}{\sigma_n}})} = \frac{ n^{e^{\frac{it}{\sigma_n}}-1}}{\Gamma(e^{\frac{it}{\sigma_n}})}(1+o(1)).
$$
As a consequence, and based on the asymptotic approximation of $\mu_n$ from \Cref{asymptotics}, we deduce from \Cref{A_n} that for $n \to \infty$:
\begin{equation}\label{An}
A_n = \exp\left(-\frac{it \log n}{\sigma_n}\right)  \frac{ n^{e^{\frac{it}{\sigma_n}}-1}}{\Gamma(e^{\frac{it}{\sigma_n}})}(1+o(1)).
\end{equation}

We further write
$$ n^{e^{\frac{it}{\sigma_n}}-1} =\exp\left(\log \left( n^{e^{\frac{it}{\sigma_n}}-1}\right) \right)= \exp\left( (e^{\frac{it}{\sigma_n}}-1) \log n  \right) .$$
The second order series expansion of the exponential, together with the asymptotic approximation of $\sigma_n^2$ from \Cref{asymptotics} yield 

$$e^{\frac{it}{\sigma_n}}-1 = \frac{it}{\sigma_n}- \frac{t^2}{2\sigma_n^2} + o\left(\frac{1}{\log n}\right),$$
and
$$n^{e^{\frac{it}{\sigma_n}}-1} = \exp\left(\frac{it \log n}{\sigma_n}-\frac{t^2}{2}\right)(1+o(1)).$$

Further replacing in \Cref{An} implies
$$
A_n =  \exp\left(-\frac{t^2}{2}\right)(1+o(1)).
$$
We have also used the fact that, by continuity, the denominator $\Gamma(e^{\frac{it}{\sigma_n}})$ converges to $\Gamma(1)=1$ as $n \to \infty$. 
Since $\varphi_n(t) = A_n + R_n$ and $R_n \longrightarrow 0$, it finally follows that $\varphi_n(t)$ converges to $e^{-t^2/2}$ as $n \to \infty$, which ends the proof.

\end{proof}

\section*{Acknowledgements}

The first author wishes to thank Fr\'ed\'eric Protin for very helpful discussions. We also wish to thank the Centre International de Rencontres Math\'ematiques (CIRM) in Marseille, who offered us the opportunity to work together during a ``research in pairs'' stay, where part of this work was performed.

\section*{References}

\bibliographystyle{elsarticle-harv}
\bibliography{prefix-exchanges-comb}

\end{document}